\documentclass[12pt,a4paper]{article}
\usepackage{fancyhdr}
\usepackage[latin1]{inputenc}
\usepackage{amsmath,amssymb,amsthm,amsfonts}
\usepackage{graphicx}
\usepackage{color}
\usepackage{mathrsfs,mathptmx}
\usepackage{graphs,graphicx,xcolor}
\usepackage{fullpage}
\usepackage[T1]{fontenc}
\newcommand{\N}{{\mathbb N}}
\newcommand{\R}{{\mathbb R}}
\newcommand{\Z}{{\mathbb Z}}
\newcommand{\uno}{{\mathbf{1}}}
\newcommand{\cero}{{\mathbf{0}}}
\newtheorem{thm}{Theorem}
\newtheorem{lemma}[thm]{Lemma}
\newtheorem{cor}[thm]{Corollary}
\newtheorem{conjecture}[thm]{Conjecture}
\newtheorem{definition}[thm]{Definition}
\newtheorem{example}[thm]{Example}
\newtheorem{rmk}[thm]{Remark}

\begin{document}

\begin{verbatim}\end{verbatim}\vspace{2.5cm}

\begin{center}
\textsc{\Large\textbf{Some advances on the set covering polyhedron of circulant matrices}}
\end{center}
\smallskip

\begin{center}
\large Silvia M.\,Bianchi\footnote{sbianchi@fceia.unr.edu.ar} \;\;\; Graciela L.\,Nasini\footnote{nasini@fceia.unr.edu.ar} \;\;\; Paola B.\,Tolomei\footnote{ptolomei@fceia.unr.edu.ar}
\end{center}

\smallskip

\begin{center}
\textit{Departamento de Matem\'atica, Facultad de Ciencias Exactas, Ingenier\'ia y Agrimensura, Universidad Nacional de Rosario, 2000 Rosario, Santa Fe, Argentina}
\medskip

\textit{and CONICET, Argentina}
\end{center}

\medskip
\begin{abstract}
Working on the set covering polyhedron of consecutive ones circulant matrices, Argiroffo and Bianchi found a class of 
facet defining inequalities, induced by a particular family of circulant minors. In this work we extend these results to 
inequalities associated with every circulant minor. We also obtain polynomial separation algorithms for particular classes of such inequalities.   

\medskip

\noindent Keywords: \textit{circulant matrix\;\; set covering polyhedron\;\; separation routines}
\end{abstract}

\section{Introduction}\label{intro}

The well-known concept of domination in graphs was introduced by Berge \cite{Be} in 1962, modelling many utility location problems in operations research. 

Given a graph $G=(V,E)$ a \emph{dominating set} is a subset $D\subset V$ such that every node outside $D$ is adjacent to at least one node in $D$. Given a cost vector $w\in \R^{|V|}$, the \emph{Minimum Weight Dominating Set Problem} (MWDSP for short), consists in finding a dominating set $D$ such that $\sum_{{v}\in D} w_{v}$ is minimum. 
MWDSP arises in many applications. We can mention the strategic placement of men or pieces on the nodes of a network. As example, consider a computer network in which one wishes to choose a smallest set of computers that are able to transmit messages to all the remaining computers \cite{24}. 
Many other interesting examples include sets of representatives, school bus routing, $(r,d)$-configurations, radio stations, social network theory, kernels of games, etc. \cite{19}.

The MWDSP is NP-hard for general graphs and has been extensively investigated from an algorithmic point of view (\cite{7,10,14,15} among others), 
The cardinality version (that is when the weights are 0 and 1) has been shown to be polynomially solvable in several classes of graphs such as cactus graphs \cite{20} and the class of series-parallel graphs \cite{23}. 

However, a few results on the MWDSP derived from the polyhedral point of view are known.
An interesting result in this context can be found in \cite{Mah}, 
working on the problem when the underling graph is a cycle.

Actually, the MWDSP corresponds to particular instances of the Mimimum Weighted Set Covering Problem (MWSCP).  

Indeed, given an $m\times n$ $0,1$ matrix $A$, a \emph{cover of} $A$ is a vector $x\in \{0,1\}^n$ such that $Ax\geq \uno$, where $\uno$ is the vector with all components at value one. 
Given a cost function $w\in \mathbb R^n$, the Mimimum Weighted Set Covering Problem (MWSCP) consists in
solving the integer program 
\[
\min\{wx:Ax\geq \mathbf 1,x\in \{0,1\}^n\}.
\]

This is equivalent to solve the problem 

\[
\min\{wx:x\in Q^*(A)\}
\]
where $Q^*(A)$ is the convex hull of points in $\{x\in \{0,1\}^n: Ax\geq \mathbf 1\}$. The set $Q^*(A)$ is usually called the \emph{set covering polyhedron} associated with $A$.

In particular, given a graph $G=(V,E)$, if $A$ is a matrix such that each row corresponds to the characteristic vector of the closed neighborhood of a node $v\in V$, (i.e., $A$ is the \emph{closed neighborhood matrix of} $G$) then 
every cover of $A$ is the characteristic vector of a dominating set of $G$ and conversely.
Therefore, solving the MWSCP on $A$ is equivalent to solve the MWDSP on $G$.

It is easy to see that the closed neighborhood matrix of a cycle is a \emph{circulant} matrix. Hence, the findings in \cite{Mah} corresponds to obtain the complete description of the set covering polyhedron for the $0,1$ $n\times n$ matrices having three consecutive ones per row, known as the family of circulant matrices $C^3_n$.

In general, the closed neighborhood  of a \emph{web} graph is a circulant matrix. Web graphs have been thoroughly studied in the literature (see \cite{Sting,Tro,Wa}).

The main goal of this work is the study of the MWSCP on circulant matrices and its direct consequences on the MWDSP when the underlying graph is a web graph. 

Previous results on the set covering polyhedron of circulant matrices can be found in \cite{Nes,circu,CN,NS,Sa}.

In section \ref{sec2} of this work, we present basic definitions and preliminaries needed for the remaining sections. In section \ref{sec3} we introduce a family of valid inequalities for  the set covering polyhedron of circulant matrices.
We obtain sufficient conditions that make a valid inequality a facet of the polyhedron. We also conjecture that this condition is also necessary. In section \ref{sec4} we prove that a subfamily of the inequalities presented in section \ref{sec3} can be separated in polynomial time.

A preliminary version of this work appeared without proofs in \cite{lagos}.

\section{Definitions, notations and preliminary results}\label{preliminar}\label{sec2}

In what follows, every time we state $S\subset \Z_n$ for some $n\in \N$, we consider $S\subset \{0,\dots,n-1\}$ and the addition between the elements of $S$ is taken modulo $n$. 

\medskip

Given a set $F$ of vectors in $\{0,1\}^n$, we say $y\in F$ is a dominating vector (of $F$) if there exits $x\in F$ such that $x\leq y$. It can be also said that $x$ is dominated by $y$. 

From now on, every matrix has $0,1$ entries, no zero columns and no dominating rows. If $A$ is such an $m\times n$ matrix, its rows and columns are indexed by $\Z_m$ and $\Z_n$ respectively. 
Two matrices $A$ and $A'$ are \emph{isomorphic} and we denote $A\approx A'$, if $A'$ can be obtained from $A$ by permutation of rows and columns.

If $S\subset \Z_m$ and $T\subset \Z_n$, let $A_{S,T}$ be the submatrix of $A$ with entries $a_{ij}$ where $i\in S$ and $j\in T$.

Given $N\subset \Z_n$, 
let us denote by $R(N)=\{ j\in\Z_m: j$ is a dominating row of $A_{\Z_m, \Z_n-N}\}$.
A \emph{minor of} $A$ \emph{obtained by contraction of} $N$ and denoted by $A/N$, is the matrix $A_{\Z_m-R(N), \Z_n-N}$. In this work, when we refer to \emph{a minor} of $A$ we are always considering a minor obtained by contraction. 

\medskip

Observe that, there exists a one-to-one correspondence between a vector $x\in \{0,1\}^n$ and the subset $S_x\subset \Z_n$ whose characteristic vector is $x$ itself.
Hence, we agree to abuse of notation by writing $x$ instead of $S_x$. In this way, if $x\in \{0,1\}^n$, we write $i\in x$ meaning that $x_i=1$. Also, if $x$ is dominated by $y\in \{0,1\}^n$ then we write $x\subset y$.

\medskip

Remind that a \emph{cover} of a matrix $A$ is a vector $x\in \{0,1\}^n$ such that $Ax\geq \uno$.
In addition, the \emph{cardinality} of a cover $x$ is denoted by $\left|x\right|$ and equals $\uno x$. A cover $x$ is \emph{minimum} if it has the minimum cardinality and in this case $\left|x\right|$ is called the \emph{covering number} of the matrix $A$, denoted by $\tau(A)$.
Observe that every cover of a minor of $A$ is a cover of $A$ and then, for all $N\subset \Z_n$, it holds that $\tau(A/N)\geq \tau(A)$.

\medskip

Recall that the set covering polyhedron of $A$, denoted by $Q^*(A)$, is defined as the convex hull of its covers. The polytope $Q(A)=\{x\in [0,1]^n: Ax\geq \uno\}$ is known as the \emph{linear relaxation of} $Q^*(A)$.
When $Q^*(A)=Q(A)$ the matrix $A$ is ideal and the MWSCP can be solved in polynomial time (in the size of $A$).

\medskip

Given $n$ and $k$ with $2\leq k\leq n-2$, 
for every $i\in \Z_n$ let $C^i=\{i,i+1,\ldots,i+(k-1)\}\subset \Z_n$.
The \emph{circulant} matrix $C_n^{k}$ is the square matrix whose $i$-th row is the incidence vector of $C^i$. 
Observe that, for $j\in \Z_n$, the $j$-th column of $C_n^k$ is the incidence vector of $C^{j-k+1}$. 

We say that a minor of $C_n^k$ is a \textit{circulant minor} if it is isomorphic to a circulant matrix.

\begin{rmk}\label{cover}
Let $C_n^k$ be a circulant matrix and let $x=\{i_j: j\in \Z_r\}\subset \Z_n$ with $0\leq i_0<i_1<\ldots<i_{r-1}\leq n-1$.
The following propositions are equivalent:
\begin{enumerate}
\item[(i)] 
$x$ is a cover of $C_n^k$,
\item[(ii)] 
$i_{j+1}-1\in C^{i_j}$ for all $j\in \Z_r$,
\item[(iii)]
$i_{j-1}\in C^{i_j-k}$ for all $j\in \Z_r$.
\end{enumerate}
\end{rmk}

It is no hard to see that $\tau(C_n^k)\geq\left\lceil \frac{n}{k}\right\rceil$. Moreover, for every $i\in \Z_n$   
$$x^i=\left\{i+hk : 0\leq h \leq \left\lfloor \frac{n}{k}\right\rfloor\right\}\subset \Z_n$$ 
is a cover of $C_n^k$ of size $\left\lceil \frac{n}{k}\right\rceil$, and then $\tau(C_n^k)=\left\lceil \frac{n}{k}\right\rceil$. 

Also, the set $\{x^i: i\in \Z_n\}$ is linearly independent if and only if $n$ is not multiple of $k$. Thus the inequality  $\sum_{i=1}^n x_i\geq \left\lceil \frac{n}{k}\right\rceil$ that is always  valid for $Q^*(C_n^k)$, defines a facet if and only if $n$ is not a multiple of $k$ (see \cite{Sa}).
This inequality will be called the \emph{rank constraint}.

In addition, for every $i\in \Z_n$, the constraints $x_i\geq 0$ and $\sum_{j\in C^i} x_j \geq 1$ are facet defining inequalities of $Q^*(C_n^k)$ (see \cite{BNg,Sa} for further details). We call them \emph{boolean facets}. 

It is also known that if $a x \geq \beta$ is a non boolean facet defining inequality of $Q^*(C_n^k)$ then $a>\cero$ \cite{circu}.

\medskip

Ideal circulant matrices have been completely identified by Cornu\'ejols et al. in \cite{CN}. Many of the ideas and results obtained in this seminal paper inspired further results presented in this work.

In fact, the authors in \cite{CN} characterize ideal circulant matrices in term of a nonideal circulant minor and give sufficient conditions for a subset $N\subset \Z_n$ to ensure that $C_n^k/N$ is a circulant minor.  These conditions are obtained in terms of simple dicycles in a particular digraph.

Indeed, given $C^{k}_n$, the digraph $G(C^{k}_n)$ has vertex set $\Z_n$ and $(i,j)$ is an arc of $G(C^{k}_n)$ if $j\in \{i+k,i+k+1\}$. In this way, we will say that an arc $(i,i+k)$ has length $k$ and an arc $(i,i+k+1)$ has length  $k+1$.

If $D$ is a simple dicycle in $G(C^{k}_n)$, and $n_2$ and $n_3$ denote the number of arcs of length $k$ and $k+1$ respectively, then there must be a positive integer  $n_1\geq 1$ such that
$n_1 n = k n_2 + (k+1) n_3$ and $\gcd(n_1,n_2,n_3)=1$ ($\gcd$ means greatest common divisor). Moreover, the conditions  $n_1 n = k n_2 + (k+1) n_3$ and $\gcd(n_1,n_2,n_3)=1$ are not only necessary but also sufficient for the existence of a simple dicycle in $G(C^{k}_n)$ (see \cite{Nescirc} for further details).

We say that $n_1, n_2$ and $n_3$ are the \emph{parameters associated with} the dicycle.

Later, Aguilera in \cite{Nes} completely characterized subsets $N$ of $\Z_n$ for which $C_n^k/N$ is a circulant minor  in terms of dicycles in the digraph $G(C^{k}_n)$. We rewrite theorem 3.10 of \cite{Nes} in the following way:

\begin{thm} \label{minorsgrafo}
Let $n,k$ be positive integers verifying $2\leq k \leq n-2$ and let $N\subset \Z_n$.
Then, the following are equivalent:
\begin{enumerate}
\item[(i)] $C_n^k/N\approx C_{n'}^{k'}$.  
\item[(ii)] $N$ induces 
$d$ disjoint simple dicycles $D_0,\ldots,D_{d-1}$ in $G(C^{k}_n)$, each of them having the same parameters $n_1$, $n_2$ and $n_3$ such that 
$n=n'-d(n_2+n_3)$ and $k'=k-dn_1$.
\end{enumerate}
\end{thm}

Thus, whenever we refer to a circulant minor of $C_n^k$ with parameters $d$, $n_1$, $n_2$ and $n_3$, we are referring to the non negative integers %
whose existence is guaranteed by the previous theorem. 
In addition, $N^j$, with $j\in \Z_d$ refers to each of the subsets inducing a simple dicycle $D^j$ in $G(C^k_n)$. Moreover,  we call $W^j=\{i\in N^j: i-(k+1)\in N^j\}$, for $j\in \Z_d$ and $W=\cup_{j\in\Z_d} W^j$. Then, $\left|W^j\right|=n_3$ and $\left|N^j\right|=n_2+n_3$ for all $j\in \Z_d$.

Observe that, the parameters $d$, $n_1$, $n_2$ and $n_3$ 
are not enough to identify the minor itself. 
For example, $C_{9}^4$ has nine different minors with parameters $d=n_1= n_2=n_3=1$. Indeed,  for every $i\in \Z_9$, $C_{9}^4/\{i,i+4\}\approx C_7^3$. 

Let us remark that starting from $W\subset \Z_n$ corresponding to a circulant minor $M$ of $C_n^k$ we can obtain the set $N\subset \Z_n$ such that $M\approx C^k_n/N$. 
In order to see this, it is enough to observe that, given $j\in W$, we can construct the set $N^j$ inducing the simple dicycle in $G(C_n^k)$ with $j\in N^j$. Indeed, let $N^j:= \{j,j-(k+1)\}$ and $i=j-(k+1)$. While $i\neq j$ we repeat the next step:
if $i\in W$ then we add $i-(k+1)$ to $N^j$ and set $i:=i-(k+1)$  else we add $i-k$ to $N^j$ and set $i:=i-k$.
Once we obtain $N^j$, it is clear that we also obtain the parameters $n_1,n_2$ and $n_3$ associated with the dicycle induced by $N^j$. Also, considering $|W|=dn_3$, we can obtain the parameter $d$. Hence, we compute $n'=n - d (n_2+n_3)$ and $k'= k-d n_1$.  

So, in what follows, we usually refer to a circulant minor
\emph{defined by}  $W\subset \Z_n$. We will also refer to \emph{the dicycle of} $G(C^k_n))$ \emph{induced by} $W^j$, considering the dicycle induced by the corresponding subset $N^j$.

\begin{rmk} \label{condW}
Let $W\subset \Z_n$.
\begin{enumerate}
	\item[(i)] If $W=\{w_i: i\in \Z_{|W|}\}$ with 
$0\leq w_0< \dots< w_{|W|-1}\leq n-1$, then $W$ defines a circulant minor with parameters $d=n_1=1$ if and only if $w_{i+1}-w_{i}=1$ (mod $k$) and $w_{i+1}-w_{i}\geq k+1$, for all $i\in \Z_{|W|}$.
	\item[(ii)] $W$ defines a circulant minor with parameters $d\geq 2$ and $n_1=1$ if and only if $W=\cup_{j\in\Z_d} W^j$, for all $j \in \Z_d$, $W^j$ defines a circulant minor with parameters $d^j=n^j_1=1$ and for all $r,j \in \Z_d$ with $r\neq j$, $N^r\cap N^j=\emptyset$. 
\end{enumerate}
\end{rmk}

Subsets $W\subset \Z_n$ that define circulant minors play an important role in the description of the set covering polytope of circulant matrices.

\section{Relevant minor inequalities}\label{minorineq}\label{sec3}

In theorem 6.9 in \cite{circu} it is proved that given a minor of $C_n^k$ with parameter $d=1$, defined by $W$ and isomorphic to $C^{k'}_{n'}$, the inequality
\[
\sum_{i\in W} 2 x_i + \sum_{i\notin W} x_i \geq \left\lceil \frac{n'}{k'}\right\rceil
\]
is a Chv\'atal-Gomory inequality of rank at most one (when starting from $Q(C_n^k)$).

Moreover, the authors proved that if $n'= 1$ (mod $k'$) and $\left\lceil \frac{n'}{k'}\right\rceil>\left\lceil \frac{n}{k}\right\rceil$, the inequality defines a facet. 

In addition, the results in \cite{Mah} imply that these inequalities, together with the boolean facets and the rank constraint completely describe the $Q^*(C^3_n)$.

\medskip

The validity of the above inequality relies on 
lemma 6.6 and corollary 6.8 in \cite{circu}. 
Actually,  the same arguments used in the proof of lemma 6.6 in \cite{circu} are enough to prove the following result:

\begin{rmk}\label{gena}
Let $N\subset\Z_n$ be such that $C_n^k/N\approx C_{n'}^{k'}$. Then $R(N)=\{i+1: i\in N\}$.
\end{rmk}

Also, in \cite{Nes} it was proved that corollary 6.8 in \cite{circu} can be extended to:

\begin{lemma}\label{genb}
Let $N\subset\Z_n$ be such that $C_n^k/N\approx C_{n'}^{k'}$. If $W=\{i\in N: i-(k+1)\in N\}$ then for all $i\in \Z_n$, 
it holds that $\left|C^{i}-N\right|=k'+1$ if $i+k\in W$ and $\left|C^{i}-N\right|= k'$ otherwise.
\end{lemma}

Thus, theorem 6.9 of \cite{circu} can be generalized in the following way:

\begin{thm}\label{validez}
Let $W\subset\Z_n$ be a subset defining a minor isomorphic to $C_{n'}^{k'}$. Then, the  inequality 
\begin{equation}\label{ecu}
\sum_{i\in W} 2 x_i + \sum_{i\notin W} x_i \geq \left\lceil
\frac{n'}{k'}\right\rceil
\end{equation}
is a valid inequality for $Q^*(C^k_n)$. Moreover, it is Chv\'atal-Gomory inequality of rank at most one (when starting from $Q(C_n^k)$).
\end{thm}

\begin{proof}
Let $N\subset \Z_n$ be the subset defining the minor i.e.  $C_n^k/N\approx C_{n'}^{k'}$ and let us call $A$ the row submatrix of $C^k_n$ defined by rows not in $R(N)$, i.e. $A=(C^k_n)_{\Z_n-R(N), \Z_n}$.
Recall that the $i$-th column  of $C_n^k$ is the incidence vector of $C^{i-k+1}$. After remark \ref{gena}, the number of entries at value one in the $i$-th column of $A$ is the number of times an index of the form $j+1$ with $j\notin N$ belongs to $C^{i-k+1}$, i.e. $\left|C^{i-k}-N\right|$. On the other hand, lemma \ref{genb} states that $\left|C^{i-k}-N\right|\in \{k',k'+1\}$ and $\left|C^{i-k}-N\right|=k'+1$ if and only if $i\in W$. 
In summary, each column of $A$ has $k'$ or $k'+1$ entries at value one. Moreover, the $i$-th column has $k'+1$ entries at value one if and only if $i\in W$. 
Thus if we add up all the rows of submatrix $A$ we get:

\begin{equation}
\label{ecuch-g}
\sum_{i\in W} (k'+1) x_i + \sum_{i\notin W} k' x_i \geq n'.
\end{equation}

Then, if we divide all the coefficients by $k'$ and round up, we obtain the inequality (\ref{ecu}). 
\end{proof}

From now on, we say that inequality (\ref{ecu}) is the \emph{minor
inequality} corresponding to the minor defined by $W$.

Remind that if $M$ is a minor of $C_n^k$ isomorphic to $C_{n'}^{k'}$ then $\left\lceil\frac{n'}{k'}\right\rceil\geq\left\lceil\frac{n}{k}\right\rceil$. Observe that when  $\left\lceil\frac{n'}{k'}\right\rceil=\left\lceil\frac{n}{k}\right\rceil$ the minor inequality is  
dominated by the rank constraint.
Also, if $n'$ is a multiple of $k'$ then it  is 
valid for $Q(C_n^k)$. 

In summary, the relevant minor inequalities correspond to minors $M$ isomorphic to $C_{n'}^{k'}$ such that 
$n'\neq 0\, (\mathrm{mod}\, k')$
and $\left\lceil\frac{n'}{k'}\right\rceil>\left\lceil\frac{n}{k}\right\rceil$.  In this case, we will say that $M$ is a \emph{relevant minor}. 

The following result identifies relevant minors:

\begin{lemma}\label{relminor}
Let $M$ be a circulant minor of $C^k_n$ isomorphic to $C_{n'}^{k'}$ with parameters $d$, $n_1$, $n_2$ and $n_3$ and let
$r$ be such that $1\leq r\leq k'-1$ and $n'=r\;$   $(\mathrm{mod}\; k')$.
Then, $M$ is a relevant minor if and only if $dn_3\geq kr$.
\end{lemma}

\begin{proof}
We know that 
$nn_1=n_2k+n_3(k+1)$, 
$n'=n-d(n_2+n_3)$ and $k'=k-dn_1$. .

Let $s$ be such that $n-d(n_2+n_3)=s(k-dn_1)+r$ then $\left\lceil\frac{n'}{k'}\right\rceil=s+1$. 

It follows that, $M$ is a relevant minor if and only if $\left\lceil\frac{n}{k}\right\rceil\leq s$. Since  
$$n=sk-\left(sdn_1-d(n_2+n_3)-r\right),$$ we have that $\left\lceil\frac{n}{k}\right\rceil\leq s$
if and only if $$sdn_1-d(n_2+n_3)-r\geq 0.$$
It is not hard to see that $$dn_3-kr= (k-dn_1)\left(sdn_1-d(n_2+n_3)-r\right).$$
Since $k-dn_1>0$, the proof is complete.
\end{proof}

Taking advantage of the same ideas in proving theorem 6.10 in \cite{circu}, we can prove the following generalization:

\begin{thm}\label{teo}

Let $W\subset\Z_n$ be a subset defining a relevant minor isomorphic to $C_{n'}^{k'}$. Then, if  $n'=1\, (\mathrm{mod} \,k')$ the  inequality 
\begin{equation}\label{ec}
\sum_{i\in W} 2 x_i + \sum_{i\notin W} x_i \geq \left\lceil \frac{n'}{k'}\right\rceil.
\end{equation}
defines a facet of $Q^*(C^k_n)$.
\end{thm}

\begin{proof}
We will show there are $n$ linearly independent roots of inequality (\ref{ec}), i.e. $n$ linearly independent covers of $C^k_n$ that satisfy (\ref{ec}) at equality.

Let $N=\cup_{j\in \Z_d} N^j \subset \Z_n$ be the subset defining the minor. i.e.  $C_n^k/N\approx C_{n'}^{k'}$ and let us denote the elements of $\Z_n - N$ as $\{v_0,\ldots,v_{n'-1}\}$ with $0\leq v_0\leq v_{1}\leq \dots \leq v_{n'-1}\leq n-1$.

Recall that, the subsets $\tilde{x}^l=\{v_{l+sk'}: 0\leq s\leq \left\lfloor \frac{n'}{k'}\right\rfloor\}$ with $l\in \Z_{n'}$ are $n'$ linearly independent minimum covers of $C^k_n/N$ and then they are $n'$ linearly independent roots of (\ref{ec}).

For the remaining $|N|$ roots, we will construct a root $z^i$ for every $i\in N$.

Observe that as $n'=1$ (mod $k'$), if $l\in \Z_{n'}$ then $l+\lfloor \frac{n'}{k'}\rfloor k'=l-1$ (mod $n'$). Hence, $v_{l-1}\in \tilde{x}^l$ for every $l\in \Z_{n'}$.
Therefore, for every $l\in \Z_{n'}$ there are two consecutive elements of $\Z_n - N$ that  belong to $\tilde{x}^l$, i.e.
$\{v_{l-1},v_l\}\subset \tilde{x}^l$ for every $l\in \Z_{n'}$. Moreover, by lemma \ref{genb}, we know that for every $i\in \Z_n$, $k'\leq |C^i-N|\leq k'+1<n'$ and then, there exists $l\in\Z_{n'}$ such that $v_{l}\notin C^i$ and $v_{l+1} \in C^i$.

Let us start with $i\in N-W$. Let $l\in \Z_{n'}$ such that $v_{l}\notin C^i$ and $v_{l+1} \in C^i$. Observe that,  by lemma \ref{genb} we have that $|C^{i-k}-N|=k'$ and since $k'\geq 2$ it follows that $v_{l-1}\in C^{i-k}$.
Then, the vector $z^i= (\tilde{x}^l -\{v_l\})\cup \{i\}$ satisfies the inequality (\ref{ec}) at equality and by the condition (iii) in remark \ref{cover} it is a cover. Also observe that $z^i\cap N=\{i\}$, and then $\{\tilde{x}^l: l\in \Z_{n'}\} \cup \{z^i:i \in N-W\}$ is a set of linearly independent covers of $C_n^k$.

Let us now obtain $z^i$ for $i\in W$. Let $i\in W$ and w.l.o.g. assume that $i\in N^0$. First, consider the minimum cover of $C^k_n$
$$x^i=\left\{i+tk: 0\leq t\leq \left\lceil \frac{n}{k} \right\rceil-1\right\}.$$
If $x^i\subset N^0$ then $x^i\cap W=\{i\}$, and since $x^i$ satisfies (\ref{ec}) we have
$$\sum_{j\in W} 2 x_j^i + \sum_{j\notin W} x_j^i=2+ \left\lceil \frac{n}{k} \right\rceil-1 \geq \left\lceil \frac{n'}{k'} \right\rceil\geq \left\lceil \frac{n}{k} \right\rceil+1.$$
Hence, $x^i$ is a root of (\ref{ec}) and then we set $z^i= x^i$.

Otherwise, let $s$ be the smallest nonnegative integer such that $s\leq\left\lceil \frac{n}{k} \right\rceil-1$ and $i+sk \notin N^0$. 
It holds that $i+sk+1\in W\cap N^0$ and for all $1\leq t\leq s-1$, $i+tk \in N^0-W$.

Now, let $l\in \Z_{n'}$ such that $v_{l-1}\notin C^i$ but $v_l \in C^i$. 

Hence, by lemma \ref{genb} we have $|C^{i+(t-1) k}-N|=k'$, for $1\leq t< s$ and $|C^{i+(s-1) k+1}-N|=k'+1$. Then,
$$C^{i+(t-1) k}-N=\{v_{l+(t-1)k'}, \dots, v_{l+tk'-1}\}$$ for all $1\leq t\leq s-1$
and $$C^{i+(s-1) k+1}-N=\{v_{l+(s-1)k'}, \dots, v_{l+sk'}\}.$$
 
We define 
$$z^i=\tilde{x}^l - (\{v_{l-1}\}\cup \{v_{l+t k'}: 0\leq t\leq s-1\})\cup \{i+t k: 0\leq t\leq s-1\}.$$
We have seen that $v_{l+sk'}\in C^{i+(s-1)k+1}$. By remark \ref{cover} (ii), we only need to prove that $v_{l+sk'}\in z^i$. For this, we need to verify that $v_{l+sk'}\neq v_{l-1}$. 

But $v_{l+sk'}=v_{l-1}$ if and only if $s=\left\lfloor \frac{n'}{k'}\right\rfloor$ and this cannot happen since we consider $s\leq\left\lceil \frac{n}{k}\right\rceil -1$ and $\left\lceil \frac{n}{k} \right\rceil < \left\lceil \frac{n'}{k'}\right\rceil$. Then $z^i$ is a cover and it is easy to check that is also a root of (\ref{ec}).
In addition $z^i\cap W=\{i\}$. Hence, it is not hard to see that the $\{ \tilde{x}^l: l\in \Z_{n'}\} \cup \{z^i:i \in N\}$ is a set of linearly independent covers of $C_n^k$.
\end{proof}

Computational experiences lead us to conjecture that the converse of theorem \ref{teo} always holds, i.e. a minor inequality defines a facet only when it corresponds to a relevant minor isomorphic to $C^{k'}_{n'}$ with  $n'=1\, (\mathrm{mod}\, k')$.

Moreover, we have the following  
\begin{conjecture}\label{resto1}
If $W\subset \Z_n$ defines a relevant minor of $C_n^k$ isomorphic to $C^{k'}_{n'}$ then, 
there exists $W'\subset W$ that defines a relevant minor isomorphic to $C^{k'}_{n''}$ such that $n''=1\, (\mathrm{mod}\, k')$ and $\left\lceil \frac{n''}{k'}\right\rceil \geq \left\lceil \frac{n'}{k'}\right\rceil$.
\end{conjecture}

Clearly, if the conjecture holds, the converse of theorem \ref{teo} is true. 
Nevertheless, we have a weaker result than the previous conjecture: 

\begin{lemma}\label{n3maschico}
If $W\subset \Z_n$ defines a relevant minor of $C_n^k$ isomorphic to $C^{k'}_{n'}$ then, 
there exists $W'\subset \Z_n$ with $|W'|\leq |W|$ that defines a relevant minor isomorphic to $C^{k'}_{n''}$ such that $n''=1\, (\mathrm{mod}\, k')$ and $\left\lceil \frac{n''}{k'}\right\rceil \geq \left\lceil \frac{n'}{k'}\right\rceil$.
\end{lemma}

\begin{proof}
Let $s$ be such that $n'=sk'+r$. If $r=1$ the result clearly holds. Let $r$ be such that $2\leq r\leq k'-1$. 

By theorem \ref{minorsgrafo} there exist non negative integers $d$, $n_1, n_2$ and $n_3$ such that $n_1 n= (n_2+n_3)k+n_3$, $n'=n- d(n_2+n_3)$, $k'=n-dn_1$ and $|W|=dn_3$.

In addition, by lemma \ref{relminor}, we have that $dn_3\geq kr$. 

Then if we set $\tilde{n}_3=dn_3-k(r-1)$ and $\tilde{n}_2=dn_2+(k+1)(r-1)$, we have that 
$n_1n= (\tilde{n}_2+\tilde{n}_3)k+\tilde{n}_3$ 
with $0<\tilde{n}_3< dn_3$. 

Considering $\tilde{d}=\gcd (n_1, \tilde{n}_2, \tilde{n}_3)$, theorem \ref{minorsgrafo} states there exists a minor of $C^k_n$ isomorphic to $C^{k'}_{n''}$ with subset $W'$ such that $|W'|=\tilde{n}_3< dn_3=|W|$ and $n''=n-(\tilde{n}_2+\tilde{n}_3)$. 

Moreover, $n''=1\,(\mathrm{mod} \, k')$ and $\left\lceil \frac{n''}{k'}\right\rceil= \left\lceil \frac{n'}{k'}\right\rceil=s+1$. \end{proof}

In addition, we can state that
\begin{lemma}\label{casopart}
The conjecture \ref{resto1} holds for relevant minors with parameters $d=n_1=1$.
\end{lemma}

\begin{proof}
Let $W\subset \Z_n$ be a subset defining a relevant minor of $C_n^k$ isomorphic to $C^{k-1}_{n'}$ and $n'=s(k-1)+r$ with $2\leq r\leq k-1$.
Assume that $W=\{w_i: i\in \Z_{|W|}\}$ with
$0\leq w_0 < w_1< \ldots < w_{|W|-1}\leq n-1$.

Take $W'=\{w_i: 0\leq i \leq |W|-k(r-1)-1\}$. Is not hard to see that, by remark \ref{condW}, $W'$ defines a relevant minor with parameters $d=n_1=1$ and by using the same arguments as in the previous lemma, the minor is isomorphic to $C^{k-1}_{n''}$ with $n''=1\, (\mathrm{mod}\,(k-1))$.
\end{proof}

As a consequence we have,

\begin{cor}
Let $k\leq 4$. If $W\subset \Z_n$ defines a relevant minor of $C_n^k$ isomorphic to $C^{k'}_{n'}$  then 
the corresponding  minor inequality defines a facet of $Q^*(C^k_n)$ if and only if $n'=1\, (\mathrm{mod} \,k')$. 
\end{cor}

\begin{proof}
If $k\leq 4$, every minor inequality valid for $Q^*(C^k_n)$ corresponds to a relevant minor isomorphic to $C^{k'}_{n'}$  with $k'=2$ or $k'=3$. If $k'=3$, the minor has parameters $d=n_1=1$ and then the corollary follows from lemma \ref{casopart}. It only remains to observe that when $k'=2$ and the minor inequality defines a facet of $Q^*(C^k_n)$, then $n'$ has to be odd.
\end{proof}

\section{The separation problem for minor inequalities}\label{sec4}

In the context of the study of the dominating set problem on cycles, the authors in \cite{Mah} give a polynomial time algorithm to separate minor inequalities valid for $Q^*(C^3_n)$. Let us observe that every circulant minor of $C^3_n$ has parameters $d=n_1=1$. 

In this section we study, the separation problem for 
inequalities associated with circulant minors of any circulant matrix with parameter $n_1=1$ and any $d\geq 1$.

In order to do so, let us first present a technical lemma for these inequalities.

\begin{lemma}\label{alfa}
Let $d, n_1=1, n_2, n_3$ be the parameters associated with a circulant minor of $C^k_n$ such that 
$n_3=r\, (\mathrm{mod} \,(k-d))$ 
with  $1\leq r< k-d$. 
Then
$$\left\lceil \frac{n-d(n_2+n_3)}{k-d}\right\rceil= \left(\frac{n}{k}-\frac{r}{k-d}+1\right)+\frac{1}{k(k-d)}dn_3. $$  
\end{lemma}

\begin{proof} 
Let $s$ be the nonnegative integer such that $n_3=s(k-d)+r$. 
Since $n=k(n_2+n_3)+n_3$ we have that 
$$\left\lceil \frac{n-d(n_2+n_3)}{k-d}\right\rceil= \left\lceil \frac{(k-d)(n_2+n_3)+n_3}{k-d}\right\rceil= n_2+n_3+s+1.$$

Since $s=\frac{n_3- r}{k-d}$ and $n_2+n_3=\frac{n-n_3}{k}$ it follows that 
$$n_2+n_3+s+1= \frac{n-n_3}{k} + \frac{n_3-r}{k-d}+1=\left(\frac{n}{k}-\frac{r}{k-d}+1\right)+\frac{1}{k(k-d)}dn_3$$
and the proof is complete.
\end{proof}

From the previous lemma, if $W\subset \Z_n$ defines a relevant minor of $C_n^k$ with parameters $d, n_1=1, n_2, n_3$ and $n_3=r\, (\mathrm{mod} \,(k-d))$ 
with  $1\leq r< k-d$, then 
the corresponding minor inequality can be written as 
$$\sum_{i\in W} x_i+\sum_{i=1}^n x_i  \geq  \alpha(d,r) +\beta(d) \left|W\right|$$ 
where 
$$\alpha(d,r)= \frac{n}{k} -\frac{r}{k-d}+1,\quad \beta(d)=\frac{1}{k(k-d)}$$
or equivalently
\begin{equation}\label{faceta}
\sum_{i\in W} (x_i - \beta(d))  \geq   \alpha(d,r)- \sum_{i=1}^n x_i.
\end{equation}

Given $C^k_n$ and two integer numbers $d,r$ with $1\leq d\leq k-2$ and $1\leq r <k-d$,  
we define the function $c^d$ on $\R$ such that $c^d(t)= t - \beta(d)$ and the function $L^{d,r}$ on $\R^n$ such that $L^{d,r}(x)= \alpha(d,r)- \sum_{i=1}^n x_i$.  

Then, the inequality (\ref{faceta}) can be written as

\begin{equation}\label{faceta2}
\sum_{i\in W} c^d (x_i) \geq   L^{d,r} (x).
\end{equation}

We will first extend to any matrix $C^k_n$ the techniques used in \cite{Mah} for matrices $C^3_n$, in order to separate inequalities corresponding to relevant minors with parameters $d=n_1=1$. 

Let us denote by $\mathcal{W}(d,r)$ the set of subsets $W\subset \Z_n$ defining relevant minors with parameters $d, n_1=1, n_2, n_3=r\, (\mathrm{mod} \, (k-d))$. 
Observe that, from lemma \ref{casopart}, when $d=n_1=1$ every relevant minor inequality  corresponds to the case $r=1$, that is why we only  consider subsets $W\in \mathcal{W}(1,1)$.

To this end, given $n,k$ 
let $K_n^k=\left(V, A\right)$ be the digraph with set of nodes 
$V=\{v^j_i :i\in \Z_{n}\;, j\in \Z_{k-1} \}\cup \{t\}$ 
and set of arcs defined as follows: first consider in $A$ the arcs  
\begin{itemize}
	\item $(v_0^0,v_l^{1})$ for all $l$ such that $k+1\leq l\leq n-1$ and $l=1$ (mod $k$),
\end{itemize}
then consider in a recursive way:
\begin{itemize}
	\item for each $(v,v_i^{j})\in A$, add $(v_i^j,v_l^{j+1})$ whenever $l$ is such that $i+k+1\leq l\leq n-1$ and $l-i=1$ (mod $k$),
  \item for each $(v,v_i^{0})\in A$, add $(v_i^0,t)$ whenever $i$ is such that $i\leq n-(k+1)$ and $n-i=1$ (mod $k$).
\end{itemize}

Note that, by construction, $K_n^k$ is acyclic. For illustration, digraph $K_{20}^4$ is depicted in figure \ref{fig:C204}.

\begin{figure}[h]
	\centering
		\includegraphics[width=0.6\textwidth]{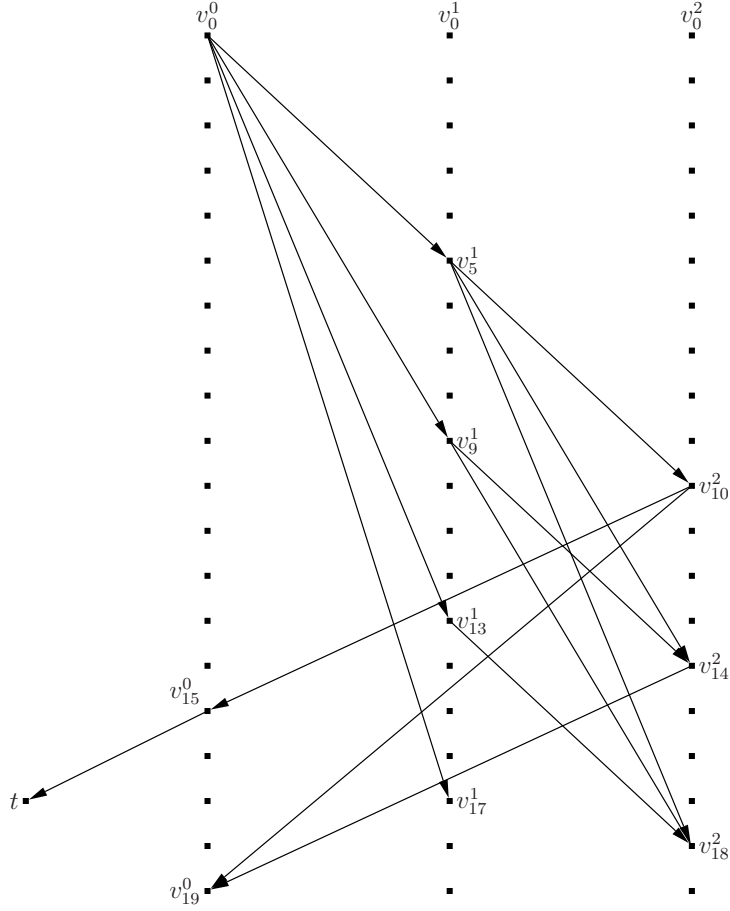}
	\caption{The digraph $K_{20}^4$}
	\label{fig:C204}
\end{figure}

We have the following result:

\begin{lemma}
There is a one-to-one correspondence between $v_0^0 t$-paths in $K_n^k$ and subsets $W\in \mathcal{W}(1,1)$ with $0\in W$.
\end{lemma}

\begin{proof}
 
Let $W\in \mathcal{W}(1,1)$ and assume that $W=\{i_j: j\in \Z_{n_3}\}\subset \Z_n$ with $0=i_0<i_1<\ldots<i_{n_3-1}\leq n-1$. Let $\alpha$ be the positive integer such that $|W|=n_3=\alpha(k-1)+1$.

Then, by remark \ref{condW} (i), $i_{j+1}- i_j= 1$ (mod $k$) and $i_{j+1}- i_j\geq k+1$ for all $j\in \Z_{n_3}$. Then,  
$$\left\{v_{i_j}^s\in V(K_n^k) \,:\, i_j\in W ,\,\, s=j\, (\mathrm{mod} \, (k-1)) \right\}\cup \left\{t\right\}$$
induces a $v_0^0 t$-path in $K_n^k$.

Conversely, let $P$ be a $v_0^0 t$-path in $K_n^k$. By construction, there exists a positive integer $\alpha$ such that 
$|V(P)\cap V^j|=\alpha$ for all $j\neq 0$ and $|V(P)\cap V^0|=\alpha+1$. 
Then, $|V(P)-\{t\}|=\alpha(k-1)+1$.

Now, if we define 
$$W=\{i\in \Z_n: v_i^j\in V(P)\;\, \mathrm{for}\; \mathrm{some } \; j\in \Z_{k-1}\}$$ then $|W|=\alpha(k-1)+1$ and from remark \ref{condW} (i) and lemma \ref{relminor}, it follows that $W\in \mathcal{W}(1,1)$. 
\end{proof}

\begin{thm}\label{sep}
Given $C^k_n$, the separation problem for inequalities corresponding to minors with parameters $d=n_1=1$ can be polynomially reduced to at most $n$ minimum weight path problems in an acyclic digraph.
\end{thm}

\begin{proof} Let $\hat{x} \in \mathbb{R}^n$. We will show that the problem of deciding if, given $j\in \Z_n$, there exists $W\in \mathcal{W}(1,1)$ with $j\in W$ and such that $\hat{x}$ violates the inequality (\ref{faceta2}) can be reduced to a shortest path problem. W.l.o.g we set $j=0$.  

Consider the digraph $K_n^k$ and associate the weight $c^1(\hat{x_i})$ with every arc $(v_l^j, v_i^{j+1})\in A$ and the weight $c^1(\hat{x_0})$ with every arc $(v_l^0, t)\in A$.

Clearly, if $W$ is the subset corresponding to a $v_0^0 t$-path $P$ in $K_n^k$, the weight of $P$ is equal to $\sum_{i\in W} c^1(\hat{x_i})$.

Then, there exists $W\in \mathcal{W}(1,1)$ with $0\in W$ and such that $\hat{x}$ violates the inequality (\ref{faceta2}) if and only if 
the minimum weight on all $v_0^0 t$-paths in $K_n^k$ is less than $L^{1,1} (\hat{x})$. 
Since $K_n^k$ is acyclic, computing this minimum weight path can be done in polynomial time using for instance Bellman algorithm \cite{bel}.
\end{proof}

\medskip

In what follows we consider inequalities corresponding to minors with parameters $n_1=1$ and $d\geq 2$. More precisely, we will focus on a particular family of minors that we call \emph{alternated minors}.

\begin{definition}\label{alternated}
Let $W=\{i_s: s\in \Z_{d n_3}\}\subset \Z_n$ with $0\leq i_0<i_1<\dots < i_{dn_3-1}\leq n-1$ be a subset defining a relevant minor of $C_n^k$ with parameters $d\geq 2,n_1=1,n_2, n_3$. 
Then, the minor defined by $W$ is a $d$-\emph{alternated minor} if, for every $j\in \Z_d$, $W^j=\{i_{j+td}: t\in \Z_{n_3}\}$.
\end{definition}

\begin{example}
In the following two cases, let us consider the circulant minors of $C^6_{33}$ induced by $W=W^0\cup W^1$: 
\begin{enumerate}
\item[(i)]
$W^0=\{7,14,21\}$ and $W^1=\{8,15,22\}$, 
\item[(ii)]
$W^0=\{7,14,21\}$ and $W^1=\{12,25,32\}$. 
\end{enumerate}
Clearly, the case (i) corresponds to a $2$-alternated minor, while the case (ii) does not.
\end{example}
  
We are interested in characterizing the subsets $W\subset \Z_n$ that define alternated minors of $C^k_n$.
 
From now on, whenever $W\subset \Z_n$ and $|W|=m$ we assume that 
$W=\{i_s: s\in \Z_{m}\}$ with $0\leq i_0<i_1<\dots < i_{m-1}\leq n-1$.
In addition, according to $W$ we let $\delta_s=i_{s+1}-i_s$, for all $s\in \Z_m$.  

\begin{rmk}\label{deltassum}
If $W\subset\Z_n$ defines a $d$-alternated minor 
it can be checked that   
$$\sum_{s=0}^{d-1} \delta_s= i_d-i_0 =1\, (\mathrm{mod}\, k) \geq k+1.$$
\end{rmk}

If in addition to the necessary condition above, we have that $\delta_{s+d}=\delta_s$ for all $s\in \Z_{|W|}$, then $W$ clearly defines a $d$-alternated minor. However, not every alternated minor verifies this condition as the following example shows:

\begin{example}\label{ej1}
Consider the minor of $C^9_{47}$ defined by $W=W^0 \cup W^1$ where $W^0=\{0,10\}$ and $W^1=\{3,22\}$. It is easy to see that $W$ defines a $2$-alternated minor but $\delta_0= 3$ and $\delta_2= 12$. 
\end{example}

Although we have:

\begin{lemma}\label{deltasig}
Let $W\subset\Z_n$ be a subset defining a $d$-alternated minor.
Then, $\delta_{s+d}=\delta_s\, (\mathrm{mod}\, k)$ for all $s\in \Z_{|W|}$.
\end{lemma}

\begin{proof}
We know that, since $W$ defines a $d$-alternated minor, for any $s\in \Z_{|W|}$, $i_s$ and $i_{s+d}$ belong to $W^j$ for some $j\in \Z_d$. Then, by remark \ref{condW} (ii), 
$i_{s+d}-i_{s}=1$ (mod $k$).

Hence, for all $s\in \Z_{|W|}$, we have that
$$\delta_{s+d}-\delta_s= \left(i_{s+1+d}-i_{s+d}\right)-\left(i_{s+1}-i_s\right)= \left(i_{s+1+d}-i_{s+1}\right)-\left(i_{s+d}-i_s\right),$$
proving that $\delta_{s+d}=\delta_s\, (\mathrm{mod}\, k)$ for all $s\in \Z_{|W|}$.
\end{proof}

We also have the following result:
\begin{lemma}\label{restos}
Let $W\subset \Z_n$ be a subset defining a $d$-alternated minor with $|W|=d n_3$. 
Then,
\begin{enumerate}
\item[(i)] if $\sum_{s=j}^r \delta_s = 0\, (\mathrm{mod}\, k)$ for some $j\leq r <j+d$, then $r=j+d-2$ and $\delta_{r+1}=1$,  	
\item[(ii)] if $\delta_s=1\, (\mathrm{mod}\, k)$ for some $s\in\Z_d$ then $\delta_{s+td}=1$ for all $t\in \Z_{n_3}$.
\end{enumerate}
\end{lemma}

\begin{proof}
In order to prove item (i), let $j,r\in \Z_{dn_3}$ with  $j\leq r <j+d$ and $\sum_{s=j}^r \delta_s=0\, (\mathrm{mod}\, k)$. Considering that  
$$\sum_{s=j}^r \delta_s=i_{r+1}-i_j,$$ 
we have that $i_{r+1}-i_j=0 \,(\mathrm{mod}\, k)$. 

Since $r+1\leq j+d$ and $i_{j+d}-i_j=1 (\mathrm{mod}\, k)$, then $r+1<j+d$ and $i_j<i_{r+1}<i_{j+d}$.  

W.l.o.g. let us assume that $j\in W^{j}$. Since $i_{j+d}\in W^{j}$, $i_{j+d}=i_j+tk+1$ for some positive integer $t$ and  $i_j+t'k \in N^j$ for all $1\leq t'\leq t-1$.  
Since $i_j<i_{r+1}<i_{j+d}$, $i_{r+1}-i_j=0 \,(\mathrm{mod}\, k)$ 
and $i_{r+1}\notin N^j$, then $i_{r+1}=i_j+tk$. 
Equivalently, $i_{r+1}=i_{j+d}-1$, $r=j+d-2$ and $\delta_{r+1}=1$.

To prove item (ii),  we only need to observe that if $\delta_s=1\, (\mathrm{mod}\, k)$ for some $s$, then using the previous lemma 
for all $t\in \Z_{n_3}$ we have,
$$\sum_{j=s+(t-1)d+1}^{s+td-1} \delta_j =0\, (\mathrm{mod}\, k),$$
and by item (i), $\delta_{s+td}=1$ for all $t\in \Z_{n_3}$. 
\end{proof}

The previous results describe necessary conditions that the values in $\{\delta_s:s\in\Z_d\}$ associated with a subset $W\subset \Z_n$ must satisfy in order to define a $d$-alternated minor. Actually, we will see that these conditions characterize these subsets.  To this purpose, let us define the following:

\begin{definition}\label{ad}
Given $k\geq 4$ and $2\leq d\leq k-2$, let $R_{d,k}\subset \Z_k^d$ such that
$(a_0,a_1, \ldots,a_{d-1})\in R_{d,k}$ if and only if 
\begin{enumerate}
\item[(i)] $\sum_{s=0}^{d-1} a_s=1\, (\mathrm{mod}\, k)\geq k+1$,
\item[(ii)]
if $\sum_{s=j}^r a_s = 0\, (\mathrm{mod}\, k)$ for some $0\leq j\leq r\leq d-1$ then $r=j+d-2$ and $j\in \{0,1\}$. 	
\end{enumerate}
\end{definition}

\begin{rmk}\label{Rdk} Observe that:
\begin{enumerate}
	\item[(i)] $R_{2,k}=\{(a_0,a_1)\in \Z_k^2:\, a_0+a_1=1\, (\mathrm{mod}\, k)\}$ and 
\item[(ii)] in general, $|R_{d,k}|=O(k^d)$.
\end{enumerate}
\end{rmk}

\smallskip

So, we have the following characterization for $d$-alternated minors.

\begin{thm}\label{alternados}
Let $d\geq 2$, $W=\{i_s: s\in \Z_{dn_3}\}\subset \Z_n$ and $W^j=\{i_{j+td}: t\in \Z_{n_3}\}$, for every $j\in \Z_{d}$.
Then, 
$W$ defines a $d$-alternated minor of $C^k_n$ if and only if there exists $a\in R_{d,k}$ such that:
\begin{enumerate}
\item[(i)]
$a_j=\delta_{j+td}\, (\mathrm{mod}\, k)$ for all $j\in \Z_d$, $t\in \Z_{n_3}$ and 
\item[(ii)]
if $a_j=1$ for some $j\in\Z_d$ then $\delta_{j+td}=1$ for all $t\in \Z_{n_3}$.
\end{enumerate}
\end{thm}

\begin{proof}
Let $W$ be a subset defining a $d$-alternated minor of $C^k_n$. For every $j\in \Z_d$, let $a_j \in\Z_k$ such that $a_j= \delta_j\, (\mathrm{mod}\, k)$. 

We first prove that $a=(a_j)_{j\in \Z_d}\in R_{d,k}$. If $d=2$, it is clear that $a=(a_0,a_1)\in R_{2,k}$.

Let $d\geq 3$. By definition of $a$ and remark \ref{deltassum}, 
$$\sum_{s=0}^{d-1} a_s=1\, (\mathrm{mod}\, k)\geq k+1$$
and condition (i) in definition \ref{ad} is verified.
Moreover, by lemma \ref{restos} (i), if $\sum_{s=j}^r a_s=0\, (\mathrm{mod}\, k)$ for some $0\leq j\leq r \leq d-1$ then $j+d=r+2$. Since $r+2\leq d+1$ then $j\in \{0,1\}$ and condition (ii) in definition \ref{ad} holds. Therefore, $a\in R_{d,k}$.

Moreover, from definition and lemma \ref{restos}, $a$ satisfies assumption (i) and (ii).

\medskip\smallskip

Conversely, let $a\in R_{d,k}$ satisfying assumptions (i) and (ii).
Since, for any $s\in \Z_{dn_3}$, $i_{s+d}-i_s= \sum_{j=s}^{s+d-1} \delta_j$, by definition \ref{ad} (i) it holds that $i_{s+d}-i_s= \sum_{j=0}^{d-1} a_j =1  \, (\mathrm{mod}\, k)$ and then $i_{s+d}-i_s=1 \, (\mathrm{mod}\, k) \geq k+1$. Then, from remark \ref{condW} (i), each $W^j$ induces a circulant minor with parameters $d=n_1=1$. 
Again from remark \ref{condW}, we only need to prove that subsets $N^j$, $j\in \Z_d$ are mutually disjoint.

Let us start with the case $d=2$. 
Suppose that there exists $v\in N^0\cap N^1$. W.l.o.g. we set  $i_1< v\leq i_2$. Then, since $v\in N^0$, $v-i_0=0$ (mod $k$) and since $v\in N^1$, $v-i_1=0$ (mod $k$). Moreover, as $i_2-i_0=1$ (mod $k$), then $i_2-v=1$ (mod $k$) and $\delta_1= i_2-i_1=(i_2-v)+(v-i_1)=1$ (mod $k$). Hence, $a_1=1$. By assumption (ii), $\delta_{1+t 2}=1$ for all $t\in \Z_{n_3}$ and it is not hard to check that, in this case, $N^0\cap N^1=\emptyset$, which is a contradiction. 

Let $d\geq 3$. W.l.o.g., it is enough to prove that $N^0\cap N^r=\emptyset$, for any $r\in \Z_d$, $r\neq 0$. To this end let $\widetilde{W}=W^0\cup W^r$. We will see that $\widetilde{W}$ defines $2$-alternated minor of $C^k_n$. 
By using the same arguments as in the case $d=2$, we only need to find $\tilde{a}=(\tilde{a}_0,\tilde{a}_1)\in R_{2,k}$ satisfying assumptions (i) and (ii) for $\widetilde{W}$ with $\tilde{\delta}_{2t}=i_{r+td}-i_{td}$ and $\tilde{\delta}_{1+2t}=i_{(t+1)d}-i_{r+td}$ for all $t\in \Z_{n_3}$. 

Let $\tilde{a}_0, \tilde{a}_1\in\Z_k$ be  such that $\tilde{a}_0=i_{r}-i_{0}\, (\mathrm{mod}\, k)$ and  $\tilde{a}_1=i_{d}-i_{r}\, (\mathrm{mod}\, k)$. Clearly, $\tilde{a}=(\tilde{a}_0,\tilde{a}_1)\in R_{2,k}$ and verifies assumption (i).

If $\tilde{a}_0=1$, $\tilde{a}_1 = 0$ i.e. 
$\sum_{i=r}^{d-1} \delta_i=0\, (\mathrm{mod}\, k)$. Then, $\sum_{i=r}^{d-1} a_i=0\, (\mathrm{mod}\, k)$.  
Hence, since $a\in R_{d,k}$ we have that $r=1$ and $\tilde{a}_0=a_0=1$. By hypothesis, $\tilde{\delta}_{2t}=\delta_{2t}=1$ for all $t\in \Z_{n_3}$.

If $\tilde{a}_1=1$, $\tilde{a}_0 = 0$, i.e. $\sum_{i=0}^{r-1} \delta_i=0\, (\mathrm{mod}\, k)$ and  $\sum_{i=0}^{r-1} a_i=0\, (\mathrm{mod}\, k)$. Hence, since $a\in R_{d,k}$, $r-1=d-2$ and $\tilde{a}_1 =a_{d-1}=1$. By hypothesis, 
$\tilde{\delta}_{1+2t}=\delta_{d-1+td}=1$ for all $t\in \Z_{n_3}$. Therefore, $\tilde{a}$ satisfies assumption (ii) and the proof is complete.
\end{proof}

Given $C^k_n$ and $d,r$ integer numbers such that $2\leq d\leq k-2$, $1\leq r\leq k-d$ we define $\mathcal{A}(d,r)$ as the set of all subsets $W\subset \Z_n$ defining a $d$-alternated minor of $C^k_n$ such that $|W|=d n_3$ with $n_3=r$ (mod $(k-d)$). Moreover, if $a\in R_{d,k}$, we define the separation problem $C_n^k$-$SP(d,r,a)$ as follows: 

\medskip

\quad\begin{minipage}{.8\textwidth}\label{psdj}\begin{tabular}{rl}
 &\\
\em INSTANCE: &
$\hat{x}\in \R^n$ \\
\em QUESTION: &
Is there $W\in \mathcal{A}(d,r)$ such that $0\in W$, \\
&$\delta_s=a_s\, (\mathrm{mod} \,k)$ for all $s\in \Z_d$ and \\
&$\sum_{i\in W} c^d(\hat{x_i})  <  L^{d,r}(\hat{x})$ ? \\
&
\end{tabular}\end{minipage}\par\medskip 

We will reduce $C^k_n$-$SP(d,r,a)$ to a shortest path problem in the digraph $K_n^k(d,r,a)$ with vertex set  $$V=\left(\bigcup_{i\in \Z_d,\, j\in \Z_{k-d+r}} V_j^i\right) \cup \{t\}$$
where
$V_j^i=\{v^i_j(p) :p\in \Z_n\}$ for all $i\in \Z_d$, $j\in \Z_{k-d+r}$. 

\medskip

The set of arcs $A$  of $K_n^k(d,r,a)$ is defined as follows: first consider in $A$ the arcs $(v_0^0(0),v_0^{1}(p))$ for all $p=a_0$ (mod $k$) and $1\leq p\leq n-1$
when $a_0\neq 1$, and $(v_0^0(0),v_0^{1}(1))$ when $a_0=1$.

Then consider in a recursive way:

\begin{itemize}	
	
\item for each $(v,v_j^{i}(p))\in A$ with $1\leq i\leq d-2$ 
	
\hspace{0.5cm}	if $a_i\neq 1$ then add $(v_j^i(p),v_j^{i+1}(q))$, for all $q$ such that  $p+a_i\leq q\leq n-1$ and $q-p=a_i$ (mod $k$), else add	$(v_j^i(p),v_j^{i+1}(p+1))$.

\item for each $(v,v_j^{d-1}(p))\in A$ with $j\leq k-d+r-2$ 

\hspace{0.5cm} if $a_{d-1}\neq 1$ then add  $(v_j^{d-1}(p),v_{j+1}^{0}(q))$,  for all $p+a_{d-1}\leq q\leq n-1$ and $q-p=a_{d-1}$ (mod $k$), else add  $(v_j^{d-1}(p),v_{j+1}^{0}(p+1))$.

\item for each $(v,v_{k-d-1}^{d-1}(p))\in A$ 
  
\hspace{0.5cm} if $a_{d-1}\neq 1$ then add $(v_{k-d-1}^{d-1}(p),v_{0}^{0}(q))$, for all $p+a_{d-1}\leq q\leq n-1$ and $q-p=a_{d-1}$ (mod $k$), 
else add $(v_{k-d-1}^{d-1}(p),v_{k-d-1}^{0}(p+1))$.
\end{itemize}

Finally, consider the following arcs: for each $(v,v_{k-d+r-1}^{d-1}(p))\in A$, if $a_{d-1}\neq 1$ then add $(v_{k-d+r-1}^{d-1}(p),t)$, for all $p\leq n-1$ and $n-p=a_{d-1}$ (mod $k$), else add  $(v_{k-d+r-1}^{d-1}(p),t)$ only when $p=n-1$.

\medskip

In figure \ref{fig:C294} we scketch the digraph $K_{29}^4(2,1,(3,2))$ where only the arcs corresponding to two $v_0^0(0) t$-paths are drawn.

\begin{figure}[h]
	\centering
		\includegraphics[width=0.37\textwidth]{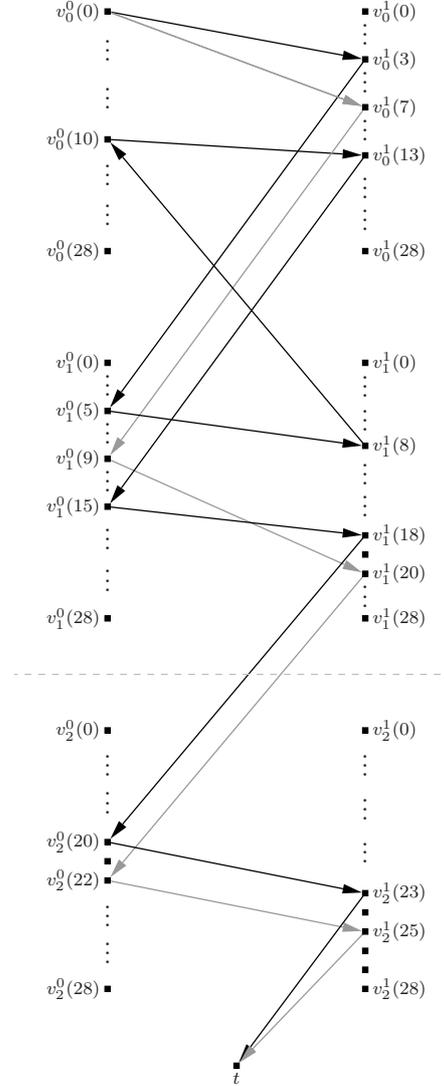}
	\caption{Two $v_0^0(0) t$-paths in the digraph $K_{29}^4(2,1,(3,2))$.}
	\label{fig:C294}
\end{figure}

Note that, by construction, if $(v_j^i (p), v_l^s(q)\in A$ then $q>p$. Hence, $K_n^k(d,r,a)$ is acyclic.

\medskip

We have the following result:

\begin{lemma}
There is a one-to-one correspondence between $v_0^0(0) t$-paths in $K_n^k(d,r,a)$ and subsets $W\in \mathcal{A}(d,r)$ with $0\in W$. 
\end{lemma}

\begin{proof}
Let $W\in \mathcal{A}(d,r)$. Then, for all $j\in\Z_d$, $W^j=\{i_{j+hd}: h\in \Z_{n_3}\}$ and $n_3=\alpha(k-d)+r$ for some positive integer $\alpha$.  

For each $h\in \Z_{n_3}$ we define $t(h)$ such that $t(h)=h\, (\mathrm{mod} \, (k-d))$ and
\begin{itemize}
	\item if $0\leq h\leq \alpha(k-d)-1$ then $t(h)\in \Z_{k-d}$ 
	\item if $\alpha(k-d)\leq h\leq \alpha(k-d)+r-1$ then $k-d\leq t(h)\leq k-d+r-1$. 
\end{itemize}

Then, we associate with every $i_{j+hd}\in W^j$, the vertex $v_{t(h)}^j(i_{j+hd})$ and 
$$\left\{v_{t(h)}^j(i_{j+hd}): j\in\Z_d, h\in\Z_{n_3}\right\}\cup \left\{t\right\}$$ induces a $v_0^0(0) t$-path in $K_n^k(d,r,a)$.

\medskip

Conversely, let $P$ be a $v_0^0(0) t$-path in $K_n^k(d,r,a)$. By construction, there exists a positive integer $\alpha$ such that 
$|V(P)\cap V^j|=\alpha(k-d)+r$ for all $j\in\Z_d$. 
Hence, if we define 
$$W^j=\{p\in \Z_n: v_i^j(p)\in V(P)\cap V^j\;\, \mathrm{for}\; \mathrm{some } \; j\in \Z_{k-d}, i\in\Z_{k-d+r}\},$$ 
then $|W^j|=\alpha(k-d)+r$. Clearly $W=\cup_{j\in\Z_d}W^j\in \mathcal{A}(d,r)$.
\end{proof}

\begin{thm} \label{sepd}
The $C_n^k$-$SP(d,r,a)$ can be polynomially reduced to a shortest path problem in a weighted acyclic digraph.
\end{thm}

\begin{proof}
Let us consider the digraph $K_n^k(d,r,a)$ and assign to every arc $(v_i^j(q), v_l^m(p))\in A$, the weight $c^d(\hat{x}_p)$  and the weight $c^d(\hat{x}_0)$ to every arc $(v_l^{d-1}(q), t)\in A$. 

Clearly, if $W$ is the subset corresponding to a $v_0^0(0) t$-path $P$ in $K_n^k(d,r,a)$, the weight of $P$ is equal to $\sum_{i\in W} c^d(\hat{x_i})$. 

Then, $\hat{x}$ violates an inequality corresponding to a circulant minor of $C^k_n$ with parameters $d$ and $n_1=1$ and subset $W$ with $0\in W$ if and only if 
the minimum weight of all $v_0^0(0) t$-paths $P$ in $K_n^k(d,r,a)$ is less than $L^{d,r} (\hat{x})$. 

Since $K_n^k(d,r,a)$ is acyclic, computing this minimum path can be done in polynomial time using for instance Bellman algorithm \cite{bel}.
\end{proof}

Finally, the separation problem for inequalities corresponding to alternated minors can be formally stated as

\medskip\quad\begin{minipage}{.8\textwidth}\label{ps}\begin{tabular}{rl}
&\\
\em INSTANCE: &
$\hat{x}\in \R^n$ \\
\em QUESTION: & Is there an alternated minor whose corresponding  \\
&inequality is violated by $\hat{x}$ ? \\
&
\end{tabular}\end{minipage}\par\medskip

Hence, from theorems \ref{sep} and \ref{sepd} and remark \ref{Rdk} (ii), we have 

\begin{thm}
For a fixed $k$, the separation problem for inequalities corresponding to alternated minors of $C^k_n$ can be solved in polynomial time.
\end{thm}

\section{Conclusions}
In this paper we study the description of the set covering polyhedron of circulant matrices. We
associate a valid inequality with each circulant minor and we show (theorem \ref{teo}) that minor inequalities include an important family of facet defining inequalities. We also give a polynomial time algorithm to separate a subfamily of them. 

As it was mentioned at the beginning of the paper, the dominating set polyhedron of web graphs is the set covering polyhedron of certain circulant matrices. Then, the results obtained so far have direct consequences on the dominating set problem on web graphs. 

We also state some interesting open questions. 

Theorem \ref{teo} gives sufficient conditions for a minor inequality in order to define a facet.  
In this way, we give a stronger conjecture (conjecture \ref{resto1}) whose validity would imply that these conditions are also necessary. 

Besides, it is known that minor inequalities completely describe the set covering polyhedron of matrices $C^3_n$. Finding new families of circulant matrices sharing this property would be a good topic for further research.

Finally, the computational complexity of the separation problem for general minor inequalities remains open.

\end{document}